\documentclass[11pt,twoside,reqno]{amsart}

\usepackage{amsmath,amsfonts,amsthm,amssymb}
\usepackage{graphicx,psfrag,overpic}

\usepackage{bm,enumerate}
\usepackage{cases}
\usepackage[all,cmtip,line]{xy}
\usepackage{array,CJK} 
\usepackage{array,setspace}
\usepackage{color}
\usepackage{yhmath}
\usepackage{verbatim}

\numberwithin{equation}{section}
\numberwithin{figure}{section}

\usepackage{graphics}
\usepackage{epsfig}
\newtheorem{claim}{\bf \t}[part]

\newtheorem{theorem}{Theorem}[section]
\newtheorem{corollary}[theorem]{Corollary}

\newtheorem{remark}{Remark}[section]
\newtheorem{definition}{Definition}[section]
\newtheorem{thm}{Theorem}[section]
\newtheorem{lem}[thm]{Lemma}

\newcommand{\set}[1]{\left\{#1\right\}}
\newcommand{\Real}{\mathbb R}
\newcommand{\abs}[1]{\left\vert#1\right\vert}

\newcommand{\qnt}[1]{\left(#1\right)}

\newcommand{\dif}{\mathrm{d}}

\title{On the one-dimensional piston model with Large Velocity Variations}
\author{Dian Hu}
\address{Dian Hu,
School of Sciences, East China University of Science and Technology, Shanghai, 200237, P.R. China}
\email{hudianaug@qq.com}

\author{Qianfeng Li}
\address{Qianfeng Li,
Department of Mathematics,
Friedrich-Alexander-Universität Erlangen-Nürnberg, Cauerstr. 11, 91058, Germany }
\email{qianfeng.li@fau.de}

\author{Yongqian Zhang}
\address{Yongqian Zhang, School of mathematical sciences, Fudan University, Shanghai 200433, P.R. China}
\email{yongqianz@fudan.edu.cn}

\begin{document}
\begin{abstract}
This paper investigates the dynamics of a one-dimensional piston expanding into a static rarefied gas. Using asymptotic analysis in the limit of vanishing initial density, we derive sharp estimates for the piston–shock distance, the separation of characteristic speeds, and the reflection coefficient associated with characteristic waves interacting with the leading shock front. Based on these estimates, we apply the method of characteristics to prove the global-in-time existence of piecewise smooth solutions. The resulting flow structure exhibits significant velocity variations. The analysis reveals a stable mechanism that operates in the vanishing-density limit of the piston model.

\end{abstract}

\maketitle
\section{Introduction}


In the paper, we consider the dynamics of the piston moving into a one-dimensional semi-infinite chamber initially filled with uniform and stationary isentropic gas. The piston model is governed by the following initial boundary value problem
\begin{equation}\label{eq:1dEulerSystem}
\left\{
\begin{aligned}
&\rho_t+(\rho u)_x=0, ~~t>0, x>w(t),\\
&(\rho u)_t+(\rho u^2+p)_x=0, ~~t>0, x>w(t),\\
&(\rho,u)=(\rho_\infty,0), ~~t=0,x>0,\\
& u(x,t)=w'(t), ~~t>0, x=w(t),
\end{aligned}
\right.
\end{equation} 
where $\displaystyle \mathsf{P}:=\set{(x, t): x=w(t), t>0}$ satisfying $w(0)=0$ denotes the piston trajectory; $u$ and $\rho$ denote the velocity and the density of the gas respectively; $\displaystyle p=\rho^{\gamma}$ denotes the pressure with adiabatic constant $\gamma\in(1, 3)$; and $\displaystyle \rho_\infty$ is a given positive constant. We aim to determine $\displaystyle (\rho, u)\big|_{t>0, x>w(t)}$ with prescribed $w(t)$ and $\rho_\infty.$

Our analysis is restricted to the compressive case, where the piston advances into the undisturbed gas, producing a leading shock front $\displaystyle \mathsf{S}:=\set{(x, t):x=s(t), s(0)=0, t>0}$. Ahead of the shock front $\mathsf{S},$ the gas is undisturbed, i.e. $(\rho,u)\big|_{t>0, x>s(t)}=(\rho_\infty,0)$. And on the shock front $\mathsf{S}$, there hold Rankine-Hugoniot (R-H) condition and entropy condition
\begin{equation}\label{eq:RHconditionExpression}
\left\{
\begin{aligned}
&(\rho-\rho_\infty)s'(t)=\rho u,\\
&\rho u s'(t) = (\rho u^2+\rho^\gamma - \rho^\gamma_\infty),\\
&\rho>\rho_\infty, \text{~on~} \mathsf{S}.  
\end{aligned}
\right.
\end{equation}
Here and in the sequel, 
\begin{equation*}
    (\rho, u)(s(t),t):=\lim_{x\to s(t)-}(\rho,u)(x,t), t>0. 
\end{equation*}

Thus, the piston problem with an advancing piston can be rewritten as 
\begin{equation}\label{eq:MainProblem}
    \left\{
    \begin{aligned}
        &\rho_t+(\rho u)_x=0,\\
        &(\rho u)_t+(\rho u^2+\rho^{\gamma})_x=0, (x,t)\in \Omega ,\\
        &u(x,t)=w'(t), (x,t)\in \mathsf{P},\\
        & \text{R-H \& entropy conditions \eqref{eq:RHconditionExpression}}, (x,t)\in \mathsf{S},
    \end{aligned}
    \right.
\end{equation}
where $\displaystyle \Omega :=\set{(x, t)|t\in\Real^+,~w(t)\leq x\leq s(t)}$, and we aim to determine $(\rho,u)\big|_{\Omega }$ and $s(t)\big|_{t>0}$.

 The piston model is not only a basic prototype model in gas dynamics \cite{Courant1948,lefranccois2010introduction}, but also an efficient approximation for hypersonic flow past slender bodies \cite{KuangJie2021,Tsien1946}. In a one-dimensional semi-infinite chamber initially filled with uniform and stationary gas, a receding piston produces a leading rarefaction wave, while an advancing piston with constant speed generates a leading shock wave with constant speed, propagating into the undisturbed medium \cite[Chapter III]{Courant1948}. Under sufficiently small bounded variation (BV) perturbations on the initial data and the piston velocity, employing the modified Glimm scheme and wave-front tracking methods,  \cite{amadori1997initial,Wang2005GlobalExistence, MR3582280} establish the stability of these two kinds of wave patterns in the framework of BV solution. We remark that the intrinsic stability mechanism is the cancellation between the leading rarefaction wave and weak shock waves of the same characteristic family, as well as the dissipation associated with the leading shock wave. Moreover, under  sufficiently small weighted $C^1$ perturbations, by the method of characteristics,  the stability of the shock wave pattern is obtained in  \cite{li1994global,li2007inverse,Li1991GlobalShock,wang2014inverse} in the framework of piecewise $C^1$ solution. 
 
Beyond the small perturbation framework, generalizing compensated compactness analysis on the Cauchy problem of 1-d isentropic gas dynamics \cite{ding1985convergenceI,ding1985convergenceII,diperna1983convergence}, Takeno \cite{takeno1992initial} establishes the global $L^\infty$ weak solution of the piston problem with any given $L^\infty$ initial data and $L^\infty$ piston velocity. Moreover, for one-dimensional free piston problems (driven solely by pressure differences), the global existence has been established both in the framework of BV solutions \cite{liu1978free} and in the sense of $L^\infty$ weak solutions \cite{takeno1995free}. In addition, when the piston velocity is near the speed of light, relativistic effects are considered in \cite{Dingmin2021,Dingmin2013,Lai2020EJAM,Laigeng2022,lai2023three}. Piston models have also been investigated in the framework of Radon measure solutions, particularly for limiting cases such as infinite-speed pistons or pressureless flows \cite{MR4052901,MR4887787}.

In the multidimensional setting, such as a spherically symmetric piston expanding into a still gas, the presence of geometric effects introduces additional analytical challenges. The self-similar flow structures are shown in \cite{Chen2003JDE,Taylor1946}. Local piecewise smooth solutions near the singular point $x=0$ have been constructed in \cite{Wangzejun2004,Wang2004ACTA}, and the global existence results of admissible BV solutions have been obtained in \cite{Wang2005DCDS,Wangzejun2004,Wangzejun2008Global}. We remark that except the flow structure inversely constructed in \cite{hu2025inverse}, all the mentioned results for multi-dimensional piston problems require almost constant piston velocity.

In the paper, we investigate the flow fields induced by the piston with large variation velocity. The main result is stated as follows:


\begin{theorem}\label{thm1}
Let $\kappa, \varrho, {w_*}, {w^*}$ be given positive constants satisfying
\begin{equation}\label{eq:WidelyVarying}
    \frac{w^*}{w_*} < 3.
\end{equation}
Then there exists $\epsilon > 0$ such that if the given initial data $(\rho_\infty, 0)$ and the piston trajectory $\mathsf{P}$ satisfy the following assumptions:
\begin{itemize}
  \item[(A1)] ${w_*} < w'(t) < {w^*}$,
  \item[(A2)] $0 < \rho_\infty < \epsilon$,
  \item[(A3)] $\displaystyle \sup_{t \in \mathbb{R}^+} |(1+t)w''(t)| < \kappa \rho_\infty^{\frac{\gamma - 1}{\gamma} + \varrho}$,
\end{itemize}
then the one-dimensional piston problem admits a global piecewise smooth solution. That is, there exist
\[
s(t) \in C^2(\mathbb{R}^+), \quad (\rho, v) \in C^1(\Omega)
\]
solving problem~\eqref{eq:MainProblem}.
\end{theorem}

\begin{remark}
The condition $0<\rho_\infty \ll 1$ in our results corresponds to a flow with a large Mach number between the piston $\mathsf{P}$ and the shock $\mathsf{S}$. Moreover, this condition can be associated with two physical regimes through appropriate scalings:

\begin{itemize}
    \item[(1)] Under the scaling
    $$
    \tilde{\rho} = \frac{\rho}{\rho_\infty}, \quad \tilde{u} = u, \quad \tilde{t} = t, \quad \tilde{x} = x,
    $$
    the limit $\rho_\infty \to 0$ corresponds to the \emph{vanishing pressure limit}. Related results on this limit process can be found in \cite{MR4466980} and the references therein.

    \item[(2)] Under the scaling
    $$
    \tilde{\rho} = \frac{\rho}{\rho_\infty}, \quad \tilde{u} = u \, \rho_\infty^{\frac{1 - \gamma}{2}}, \quad \tilde{t} = t \, \rho_\infty^{\frac{\gamma - 1}{2}}, \quad \tilde{x} = x,
    $$
    the regime $0<\rho_\infty \ll 1$ corresponds to the situation where the piston moves into a still gas at very high speed.
\end{itemize}
\end{remark}

    
    
   


The formula \eqref{eq:WidelyVarying} and the assumption \textup{(A1)} imply that the piston exhibits large velocity variations in our setting. Consequently, $u\big|_{\Omega }$ varies significantly. Thus, a perturbation approach assuming $(\rho,u)\big|_{\Omega }$ is near a constant state, as in \cite{Chen2003JDE}, is not applicable here. To overcome this,  we assume sufficiently small initial density, i.e., $0<\rho_\infty\ll1$, which ensures $\Omega $ is such a narrow dihedral region. Indeed, under assumptions \textup{(A1)} - \textup{(A3)}, there holds $$\displaystyle \big|s(t)-w(t)\big|=\mathcal{O}(\rho_\infty^{\frac{\gamma-1}{\gamma}})t.$$ 

A similar strategy has been employed in \cite{cui2009global,hu2018supersonic,HuZhang2019SIMA} to address the challenge induced by large variations in flow velocity, thereby enabling the construction of a global piecewise smooth solution for stationary hypersonic irrotational flow past sharp obstacles. However, in the piston problem, the inherent structure of the nonstationary isentropic Euler equations introduces two additional challenges:
\begin{itemize}
    \item  There is a loss of strict hyperbolicity of the nonstationary isentropic Euler equations in $\Omega $ as $\rho_{\infty}$ tends to $0.$ Indeed, the differences among eigenvalues, the shock speed and the piston speed are $\displaystyle\mathcal{O}(\rho_\infty^{\frac{\gamma-1}{2\gamma}}).$

    \item The strict dissipation property of leading shock wave degenerates as $\displaystyle\rho_\infty\to 0$. Specifically, the reflection coefficient associated with characteristic waves interacting with the leading shock front at $(s(t),t)\in \mathsf{S}$ is almost $ 1-\frac{6}{\sqrt{\gamma}}w'(t)^{-\frac{1}{\gamma}}\rho_\infty^{\frac{\gamma-1}{2\gamma}}.$
\end{itemize}
Consequently, it is necessary to balance the required smallness of $\rho_{\infty}$—which serves to shorten the distance between the shock wave and the piston—against the resulting degeneration of both strict hyperbolicity and the shock wave’s strict dissipation property.



Specifically, inspired by the asymptotic behavior of the flow field near $\rho_\infty=0$ with constant piston velocity, we design suitable a priori assumptions such that the variations of the flow states and their derivatives can be dissipated by the leading shock wave. Moreover, restricting the variation produced at the piston to match the dissipation at the leading shock wave, we can complete the whole a priori estimates. We remark that there hold in our obtained flow field,  $$\bigg|\rho\big|_{t=t_0,w(t_0)<x<s(t_0)}-\rho\big|_{t=t_0,x=s(t_0)}\bigg|\leq \rho_\infty^{\frac{\gamma+1}{2\gamma}}, $$ and $$\bigg|u\big|_{t=t_0,w(t_0)<x<s(t_0)}-w'(t_0)\bigg|\leq (w'(t)^{\frac{\gamma-2}{\gamma}}+\delta_1)\rho_\infty^{\frac{\gamma-1}{\gamma}},$$ for some positive constant $\delta_1$ determined in Lemma \ref{lem:Apriori}. 

The remaining part is organized as follow. In Section 2, we establish the asymptotic behavior of the flow states near $\rho_\infty=0$ in the case piston moving with constant speed. In Section 3, we establish the characteristic decomposition, and figure out reflection relations at the piston and at the shock wave. In Section 4, based on the characteristic forms and the dissipation property of the shock wave, by conducting analysis along characteristics, we complete the a priori estimates. Moreover, as a corollary of the a priori estimates, we prove the main Theorem \ref{thm1}.

\section{Piston moving with constant speed}
We consider a special case in which the piston expands into still gas $(\rho_\infty, 0)$ with constant speed $w_0$.  According to \cite[Chapter III]{Courant1948}, the resulting flow field consists of piecewise constant states separated by a straight shock moving at constant speed $s_0$. The flow field ahead of the shock is $(\rho_\infty, 0)$ and the flow field behind the shock is $(\rho_0, u_0).$ Here constant speed $s_0$ and the state $(\rho_0, u_0)$ are determined by Rankine-Hugoniot condition and entropy condition. That is, in this case, the problem reduces to finding $s_0$ and $(\rho_0,u_0)$ from the following 
\begin{equation}\label{eq:constantRHC}
\begin{cases}
(\rho_0-\rho_\infty)s_0=\rho_0u_0,\\
\rho_0 u_0 s_0=\rho_0 u_0^2+\rho^\gamma_0-\rho^\gamma_\infty,\\
\rho_0>\rho_\infty.
\end{cases}
\end{equation}
It is obvious that the solid wall boundary condition at piston leads to the following 
\begin{equation}\label{eq:constantBC}
u_0=w_0.
\end{equation}

Before proceeding further, we define the equivalent infinitesimals as follows. 

\begin{definition}\label{def:equilinfinitesmall}
     Let $\mathcal{F}, \mathcal{G}$ denote positive quantities associated with $(x,t)$ and $\rho_\infty.$ We say $\mathcal{F}, \mathcal{G}$ are equivalent infinitesimals, denoted as $\mathcal{F}\sim\mathcal{G}$, if and only if $$\lim_{\rho_\infty\to0}\mathcal{F}=\lim_{\rho_\infty\to0}\mathcal{G}=0, ~ \lim_{\rho_\infty\to0+}\mathcal{F}/\mathcal{G}=1,$$ 
    and the limits are uniform with $(x,t)\in\mathbb{R}^+\times\mathbb{R}^+$.
\end{definition}

We next study the solvability of \eqref{eq:constantRHC}-\eqref{eq:constantBC}, and figure out the asymptotic behavior of $(\rho_0, u_0)$ and $s_0$ as $\rho_\infty$ approaches $0$.

\begin{thm}\label{thm:ConstantCase}
For given positive constants $\rho_\infty$ and $w_0$, there are unique $(\rho_0, u_0)$ and $s_0$ solving \eqref{eq:constantRHC}-\eqref{eq:constantBC}. Furthermore, it holds that  
\begin{equation}
    \rho_0\sim w_0^{\frac{2}{\gamma}}\rho^{\frac{1}{\gamma}}_\infty, ~u_0=w_0, ~s_0-w_0\sim w_0^{\frac{\gamma-2}{\gamma}}\rho^{\frac{\gamma-1}{\gamma}}_\infty.
\end{equation}
\end{thm}
\begin{proof}
It follows from \eqref{eq:constantRHC}-\eqref{eq:constantBC} directly that 
\begin{equation}\label{eq:RHconditionC}
\begin{cases}
w_0=\sqrt{\frac{(\rho_0-\rho_\infty)(\rho_0^{\gamma}-\rho^{\gamma}_\infty)}{\rho_\infty\rho_0}},\\
w_0=\frac{(\rho_0-\rho_\infty)}{\rho_0}s_0,\\
\rho_0>\rho_\infty.
\end{cases}
\end{equation}
Setting $\displaystyle\tau={\rho_0}/{\rho_\infty},$ we rewrite \eqref{eq:RHconditionC} into 
\begin{equation}\label{eq:ReformulationofE1}
\begin{cases}
w^2_0\rho^{1-\gamma}_\infty=\frac{(\tau-1)(\tau^{\gamma}-1)}{\tau}:=h(\tau),\\
s_0=\frac{\tau}{\tau-1}w_0,\\
\tau>1.
\end{cases}
\end{equation}
Noting  that $h(\tau)$ is smooth and strictly increasing for $\tau>1$, and that 
\begin{equation}\label{eq:limitoff}
\lim_{\tau\to 1}h(\tau)=0, \lim_{\tau\to+\infty}h(\tau)=+\infty,
\end{equation}
we conclude that for given positive constants $w_0$ and $\rho_\infty,$ there exists unique $\tau$ solving the first equation of \eqref{eq:ReformulationofE1}. That is, \eqref{eq:RHconditionC} admits unique $\rho_0$. Then, $s_0$ is uniquely given by the second equation of \eqref{eq:ReformulationofE1}.   


Since $h(\tau)$ is monotonically increasing for  $\tau>1$, it follows from the first equation of $\eqref{eq:ReformulationofE1}$ that 
\begin{equation*}
\lim_{\rho_\infty\to0+}\tau=+\infty,
\end{equation*}
which further together with \eqref{eq:ReformulationofE1} implies that 
\begin{equation}\label{eq:Asympoft}
\tau\sim w_0^{\frac{2}{\gamma}}\rho_\infty^{\frac{1-\gamma}{\gamma}}. 
\end{equation}

Noting that $\displaystyle\tau={\rho_0}/{\rho_\infty}$, a direct computation shows that 
\begin{equation*}
\rho_0\sim w_0^{\frac{2}{\gamma}}\rho_\infty^{\frac{1}{\gamma}}, ~s_0-w_0\sim w_0^{\frac{\gamma-2}{\gamma}}\rho_\infty^{\frac{\gamma-1}{\gamma}}.
\end{equation*}
The proof is complete.
\end{proof}

As a consequence of Theorem \ref{thm:ConstantCase}, we have the following corollary.

\begin{corollary}
For $(\rho, u)=(\rho_0, u_0)$ and $s(t)=s_0t$, it holds that 
\begin{equation}
\lambda_+(s(t),t)-s_0\sim s_0-\lambda_-(s(t),t)\sim c(s(t),t)\sim \sqrt{\gamma}w_0^{\frac{\gamma-1}{\gamma}}\rho_\infty^{\frac{\gamma-1}{2\gamma}}.
\end{equation}
where the sound speed $c$ and the eigenvalues $\lambda_\pm$ are defined in \eqref{eq:Defofc} and \eqref{eq:Characteristics}. 
\end{corollary}

\section{characteristic decomposition and its application}

Following the discussions in \cite{LiYangZheng2011JDE,LiZhangZheng2006CMP,LiZheng2009ARMA}, we present the characteristic decomposition of the one-dimensional isentropic compressible Euler equations. This decomposition facilitates the analysis of reflection of characteristic waves at the leading shock and the piston. 

\subsection{Characteristic Decomposition}
We denote by 
$$\Omega_\mathsf{T}:=\{t<\mathsf{T}\}\cap\Omega,~~  \mathsf{T} \in \mathbb{R}^+,$$
the portion of the domain $\Omega$ restricted to the time interval $[0,\mathsf{T}).$  
Considering the smooth flow field in $\Omega_{\mathsf{T}}$ with positive density, we rewrite the first two equations in \eqref{eq:MainProblem} into the non-conservation form 
\begin{equation}\label{eq:1DEulerMatrix}
\begin{pmatrix}
\rho \\
u
\end{pmatrix}_t
+
\begin{pmatrix}
u & \rho \\
\frac{c^2}{\rho} & u
\end{pmatrix}
\begin{pmatrix}
\rho \\
u
\end{pmatrix}_x=
\begin{pmatrix}
0 \\
0
\end{pmatrix},
\end{equation}
where the sound speed $c$ is 
\begin{equation}\label{eq:Defofc}
c:=\sqrt{p'(\rho)}=\sqrt{\gamma}\rho^{\frac{\gamma-1}{2}}.
\end{equation}
Direct computation shows that system \eqref{eq:1DEulerMatrix} is hyperbolic  with eigenvalues 
\begin{equation}\label{eq:Characteristics}
\lambda_\pm := u \pm c,
\end{equation}
and corresponding left eigenvectors
\begin{equation}\label{eq:lvecter1DEulerMatrix}
l_+ = (c, \rho),~ l_- = (-c, \rho).
\end{equation}

Multiplying \eqref{eq:1DEulerMatrix} by left eigenvectors, we obtain that 
\begin{equation}\label{eq:1DEulerCharacteristicForm}
\partial^+ u=-\frac{2}{\gamma-1} \partial^+ c,~ \partial^- u=\frac{2}{\gamma-1} \partial^- c,
\end{equation}
where
\begin{equation}\label{eq:CharacteristicsDerivatives}
\partial^\pm:=\partial_t+\lambda_\pm\partial_x.
\end{equation}
We also note that for positive $\rho$, $\partial_x$ and $\partial_t$ can be expressed in terms of $\partial^\pm$ as follows:
\begin{equation}\label{eq:CharacteristicExpressionsX}
\partial_x=\frac{\partial^+-\partial^-}{2c}, ~ \partial_t=\frac{\lambda_+\partial^--\lambda_-\partial^+}{2c}.
\end{equation}

Based on the preceding notation and relations, we now present the characteristic decomposition:
\begin{thm}[Characteristic decomposition]
Suppose that the flow field in $\Omega_{\mathsf{T}}$ is smooth. Then there holds in $\Omega_{\mathsf{T}}$ that
\begin{equation}\label{eq:CharacteristicDecomposition}
\begin{cases}
\displaystyle\partial^+\partial^-c=\frac{\gamma+1}{\gamma-1}\frac{1}{2c}\qnt{\partial^+c+\partial^-c}\partial^-c,\\
\displaystyle\partial^-\partial^+c=\frac{\gamma+1}{\gamma-1}\frac{1}{2c}\qnt{\partial^+c+\partial^-c}\partial^+c.
\end{cases}
\end{equation}
\end{thm}

\begin{proof}
For any smooth function $I(x,t)$, we have the Lie relation
\begin{equation}\label{eq:LieRelation}
\begin{split}
\partial^+\partial^-I-\partial^-\partial^+I&=\qnt{\partial^+u-\partial^+c}\partial_xI-\qnt{\partial^-u+\partial^-c}\partial_xI\\
&=\qnt{\partial^+u-\partial^+c}\frac{\partial^+I-\partial^-I}{2c}-\qnt{\partial^-u+\partial^-c}\frac{\partial^+I-\partial^-I}{2c}\\&=-\frac{\gamma+1}{\gamma-1}\frac{1}{2c}\qnt{\partial^+c+\partial^-c}\qnt{\partial^+I-\partial^-I},
\end{split}
\end{equation}
where \eqref{eq:1DEulerCharacteristicForm} is used to eliminate $\partial^{\pm}u$. 

By setting $I=c$ in \eqref{eq:LieRelation}, we obtain
\begin{equation}\label{eq:LieRelationC}
\partial^+\partial^-c-\partial^-\partial^+c=-\frac{\gamma+1}{\gamma-1}\frac{1}{2c}\qnt{\partial^+c+\partial^-c}\qnt{\partial^+c-\partial^-c}.
\end{equation}

Next, setting $I=u$ and applying \eqref{eq:1DEulerCharacteristicForm} to eliminate $\partial^\pm u$, a direct computation yields that 
\begin{equation}\label{eq:LieRelationC2}
\partial^+\partial^-c+\partial^-\partial^+c=\frac{\gamma+1}{\gamma-1}\frac{1}{2c}\qnt{\partial^+c+\partial^-c}^2.
\end{equation}

Combining \eqref{eq:LieRelationC} with \eqref{eq:LieRelationC2}, we arrive at \eqref{eq:CharacteristicDecomposition}. The proof is complete.
\end{proof}

\subsection{Reflections at the Leading Shock Front and the Piston}
Define
\begin{equation} \mathsf{k}(t) = \frac{\rho(s(t), t)}{\rho_\infty}, ~ t > 0. 
\end{equation}
Then, the entropy condition implies
\begin{equation} \mathsf{k}(t) > 1, ~ t > 0. 
\end{equation}
Furthermore, the Rankine–Hugoniot condition implies that
\begin{equation}\label{eq:ShockPolarURho}
u(s(t),t)=\rho_\infty^{\frac{\gamma-1}{2}}\sqrt{\qnt{1-\frac{\rho_\infty}{\rho(s(t),t)}}\qnt{\qnt{\frac{\rho(s(t),t)}{\rho_\infty}}^\gamma-1}}=\rho_\infty^{\frac{\gamma-1}{2}}f(k)\Big|_{k=\mathsf{k}(t)},
\end{equation}
and 
\begin{equation}\label{eq:RHconditionUS}
u(s(t),t)=s'(t)(1-\frac{1}{k})\Big|_{k=\mathsf{k}(t)}.
\end{equation}
Here and in the sequel, the function $\displaystyle f(k), k>1$ is defined by 
\begin{equation}\label{eq:fExpression}
f(k)=\sqrt{\qnt{1-\frac{1}{k}}\qnt{k^\gamma-1}}, ~k>1.
\end{equation}

As a consequence of \eqref{eq:ShockPolarURho} and \eqref{eq:RHconditionUS}, a direct computation yields that
\begin{equation}\label{eq:RHconditionSRho}
s'(t)=\rho_\infty^{\frac{\gamma-1}{2}}\frac{kf(k)}{k-1}\Big|_{k=\mathsf{k}(t)},
\end{equation}
which together with \eqref{eq:CharacteristicExpressionsX} yields the expression of the differential operator along the shock wave front $\partial^s$ in terms of $\partial^{\pm}$ as follows:
\begin{equation}\label{eq:SDerivativeExpression}
\partial^s:=\partial_t+s'(t)\partial_x=a(t)\partial^++b(t)\partial^-,
\end{equation}
where 
\begin{equation}\label{eq:abExpressions}
a(t)=\frac{f(k)+\sqrt{\gamma}(k-1)k^{\frac{\gamma-1}{2}}}{2\sqrt{\gamma}(k-1)k^{\frac{\gamma-1}{2}}}\Big|_{k=\mathsf{k}(t)}, ~
b(t)=\frac{\sqrt{\gamma}(k-1)k^{\frac{\gamma-1}{2}}-f(k)}{2\sqrt{\gamma}(k-1)k^{\frac{\gamma-1}{2}}}\Big|_{k=\mathsf{k}(t)}.
\end{equation}

Now, we are ready to give the transformation relation between $\partial^{\pm} c$ on shock $\mathsf{S}$.  

\begin{thm}[Reflection at shock front] \label{lem:R3}
Suppose that the flow field is smooth in $\Omega_{\mathsf{T}}$. Then, there holds that over shock  front $\mathsf{S}\cap\{t<\mathsf{T}\}$,  
\begin{equation}\label{eq:ShockRelations}
\partial^+c+k_g\partial^-c=0,
\end{equation}
where
\begin{equation}
k_g= \frac{\qnt{\sqrt{\gamma}(k-1)k^{\frac{\gamma-1}{2}}-f(k)}}{\qnt{f(k)+\sqrt{\gamma}(k-1)k^{\frac{\gamma-1}{2}}}}\frac{\frac{2}{\sqrt{\gamma}}f'(k)k^{\frac{3-\gamma}{2}}-1}{\frac{2}{\sqrt{\gamma}}f'(k)k^{\frac{3-\gamma}{2}}+1}\Big|_{k=\mathsf{k}(t)},
\end{equation}
and
\begin{equation*}
f'(k)=\frac{\gamma k^{\gamma+1}+(1-\gamma)k^{\gamma}-1}{2k^2f(k)}.
\end{equation*}

Furthermore, suppose that 
\begin{equation*}
\lim_{\rho_{\infty\to 0+}}\mathsf{k}(t)=+\infty,
\end{equation*}
and the limit is uniform with respect to $t\in \mathbb{R}^+.$ Then, it holds that 
\begin{equation}\label{eq:kgExpan}
1-k_g \sim \frac{6}{\sqrt{\gamma}}(\mathsf{k}(t))^{-\frac{1}{2}},
\end{equation}
where the notation $\sim$ is defined in Definition \ref{def:equilinfinitesmall}. 

\end{thm}
\begin{proof}
Differentiating \eqref{eq:ShockPolarURho} with respect to $t$, and substituting the definition of $\mathsf{k}(t)$ and the sound speed $c$, we obtain that at the shock wave front $\mathsf{S}$,
\begin{equation}\label{eq:EE2}
    \begin{aligned}
        &\partial^su=\rho_\infty^{\frac{\gamma-3}{2}}f'(k)\partial^s\rho\Big|_{k=\mathsf{k}(t)}=\rho_\infty^{\frac{\gamma-3}{2}}f'(k)\frac{2}{\sqrt{\gamma}(\gamma-1)}\rho^{\frac{3-\gamma}{2}}\partial^sc\Big|_{k=\mathsf{k}(t)}\\&= \frac{2}{\sqrt{\gamma}(\gamma-1)}f'(k)k^{\frac{3-\gamma}{2}}\partial^sc\Big|_{k=\mathsf{k}(t)}.
    \end{aligned}
\end{equation}
Then, Plugging \eqref{eq:SDerivativeExpression} into the former \eqref{eq:EE2}, we obtain that at the shock wave front $\mathsf{S}$,
\begin{equation}\label{eq:EE3}
\qnt{a(t)\partial^++b(t)\partial^-}u= \frac{2}{\sqrt{\gamma}(\gamma-1)}f'(k)k^{\frac{3-\gamma}{2}}\qnt{a(t)\partial^++b(t)\partial^-}c\Big|_{k=\mathsf{k}(t)},
\end{equation}
where $a(t)$ and $b(t)$ are given in \eqref{eq:abExpressions}. 

Moreover, due to \eqref{eq:1DEulerCharacteristicForm}, we eliminate $\partial^\pm u$ in \eqref{eq:EE3} to get that at the shock wave front $\mathsf{S}$,
\begin{equation}
\partial^+c+\frac{\frac{2}{\sqrt{\gamma}}f'(k)k^{\frac{3-\gamma}{2}}-1}{\frac{2}{\sqrt{\gamma}}f'(k)k^{\frac{3-\gamma}{2}}+1}\frac{b(t)}{a(t)}\partial^-c\Big|_{k=\mathsf{k}(t)}=0,
\end{equation}
which together with \eqref{eq:abExpressions} gives \eqref{eq:ShockRelations}. 

Finally, direct asymptotic analysis on \eqref{eq:ShockRelations} yields \eqref{eq:kgExpan}. The proof is complete. 
\end{proof}

We next derive the transformation relation between $\partial^{\pm}c$ over the piston trajectory $\mathsf{P}$.

\begin{thm}[Reflection at piston]
    Suppose that the flow field is smooth in $\Omega_{\mathsf{T}}$. Then, there holds that over piston trajectory $\mathsf{P} \cap \{t<\mathsf{T}\}$,  
\begin{equation}\label{eq:pistonrelations}
\partial^+c-\partial^-c=\qnt{1-\gamma}w''(t).
\end{equation}
\end{thm}
\begin{proof}
    Differentiating the the third equation in \eqref{eq:MainProblem} with respect to $t,$ and substituting \eqref{eq:CharacteristicExpressionsX},
we have that at the piston trajectory $\mathsf{P}$, 
\begin{equation}
w''(t)=\partial_t u + u \partial_x u=
\frac{1}{2}\qnt{\partial^+u + \partial^-u},
\end{equation}
which together with \eqref{eq:1DEulerCharacteristicForm} gives \eqref{eq:pistonrelations}.
\end{proof}

\section{Proof of Theorem \ref{thm1}}

Based on the characteristic decomposition of compressible Euler equations and the dissipation at the shock front, we establish the following continuation criterion. There exist positive constants $\delta_1, \delta_2$ independent of $\mathsf{T}$ such that if 
      \begin{equation*}
(\mathcal{H})
\begin{cases}
\abs{\rho(x, t)-\rho(s(t), t)}\leq \rho_\infty^{\frac{\gamma+1}{2\gamma}},\\
\abs{s'(t)-u(x, t)}\leq (w'(t)^{\frac{\gamma-2}{\gamma}}+\delta_1)\rho_\infty^{\frac{\gamma-1}{\gamma}},\\
\max\set{\abs{t\partial^+c}, \abs{t\partial^-c}}\leq \delta_2\rho_\infty^{\frac{\gamma-1}{2\gamma}},
\end{cases}
\end{equation*}
holds over $\Omega_\mathsf{T}$, then a sharper estimate $(\tilde{{\mathcal{H}}})$ stated in Lemma \ref{lem:Apriori} is valid in the same domain. Consequently, the flow field can be extended from $\Omega_\mathsf{T}$ to $\Omega_{\mathsf{T+h}}$ with $(\mathcal{H})$ for some $\mathsf{h} > 0$ independent of $\mathsf{T}$ with $(\mathcal{H})$ remaining valid.


Specifically, as $\rho_\infty\to0+,$ the condition $(\mathcal{H})$ ensures that the leading part of $(\rho(x, t_0), u(x, t_0))$ matches the background solution corresponding to a piston moving at constant speed $w'(t_0)$. We clarify this in the following lemma.

\begin{lem}\label{lem:perturbationfirmula} 
Let $T$ be a positive constant. Suppose that \textup{(A1)} holds, and that $\displaystyle\qnt{\rho, u}\big|_{\Omega_T} \in C^1(\Omega_T)$ satisfies $(\mathcal{H})$. Then it holds that for $(x,t)\in\Omega_T,$ 
\begin{equation}\label{eq:perturbedFormula}
\begin{cases}
\rho(x, t)\sim w'(t)^{\frac{2}{\gamma}}\rho_\infty^{\frac{1}{\gamma}},\\
c(x, t)=\sqrt{\gamma}\rho(x, t)^{\frac{\gamma-1}{2}}\sim \sqrt{\gamma}w'(t)^{\frac{\gamma-1}{\gamma}}\rho_\infty^{\frac{\gamma-1}{2\gamma}},\\
\lambda_{+}(x, t)-s'(t)\sim \sqrt{\gamma}w'(t)^{\frac{\gamma-1}{\gamma}}\rho_\infty^{\frac{\gamma-1}{2\gamma}},\\
w'(t)-\lambda_-(x, t)\sim \sqrt{\gamma}w'(t)^{\frac{\gamma-1}{\gamma}}\rho_\infty^{\frac{\gamma-1}{2\gamma}},\\
1-k_g \sim \frac{6}{\sqrt{\gamma}}w'(t)^{-\frac{1}{\gamma}}\rho_\infty^{\frac{\gamma-1}{2\gamma}}.
\end{cases}
\end{equation}

Furthermore,  there exist maps 
\begin{equation*}
\begin{split}
\mathcal{E}:~~(0, \sqrt{\gamma}{w_*}^{\frac{\gamma-1}{\gamma}})&\longrightarrow \Real^+\\
\sigma&\longrightarrow \mathcal{E}(\sigma)
\end{split}
\end{equation*}
and
\begin{equation*}
\begin{split}
\mathcal{W}:~~(0, \frac{6}{\sqrt{\gamma}}{w^*}^{-\frac{1}{\gamma}})&\longrightarrow \Real^+\\
\sigma&\longrightarrow \mathcal{W}(\sigma)
\end{split}
\end{equation*}
such that for $(x,t)\in\Omega_T,$ when $\rho_\infty\in(0, \mathcal{E}(\sigma)),$ 
\begin{equation}\label{eq:lowerboundedonH}
\begin{cases}
\lambda_+(x, t)-s'(t)>\qnt{\sqrt{\gamma}{w_*}^{\frac{\gamma-1}{\gamma}}-\sigma}\rho^{\frac{\gamma-1}{2\gamma}}_\infty,\\
w'(t)-\lambda_-(x, t)>\qnt{\sqrt{\gamma}{w_*}^{\frac{\gamma-1}{\gamma}}-\sigma}\rho^{\frac{\gamma-1}{2\gamma}}_\infty,\\
\qnt{\sqrt{\gamma}{w^*}^{\frac{\gamma-1}{\gamma}}+\sigma}\rho^{\frac{\gamma-1}{2\gamma}}_\infty> c(x, t)> \qnt{\sqrt{\gamma}{w_*}^{\frac{\gamma-1}{\gamma}}-\sigma}\rho^{\frac{\gamma-1}{2\gamma}}_\infty,
\end{cases}
\end{equation}
and when $\rho_\infty\in (0, \mathcal{W}(\sigma)),$
\begin{equation}\label{eq:kgupperbound}
\abs{k_g}\leq 1- \qnt{\frac{6}{\sqrt{\gamma}}{w^*}^{-\frac{1}{\gamma}}-\sigma}\rho_\infty^{\frac{\gamma-1}{2\gamma}}.
\end{equation}
\end{lem}

\begin{proof}
Due to the second formula in $(\mathcal{H}),$ a direct computation shows that 
\begin{equation}\label{eq:G2}
\begin{aligned}
    &\abs{u_S(t) - w'(t)}\\&\leq\abs{u_S(t)-s'(t)}+\abs{u(t, w(t))-s'(t)}\\
&\leq 2(w'(t)^{\frac{\gamma-2}{\gamma}}+\delta_1)\rho_{\infty}^{\frac{\gamma-1}{\gamma}}.
\end{aligned}
\end{equation}
Here and in the sequel, for simplification, we set 
\begin{equation}
\rho_S(t):=\lim_{x\to s(t)-}\rho(x, t),~ u_S(t):=\lim_{x\to s(t)-}u(x, t).
\end{equation}
Plugging the boundedness of $w'(t)$ in \textup{(A1)} into \eqref{eq:G2}, we conclude that for sufficiently small positive $\rho_\infty$, $u_S(t), t>0$ are in some compact subset of $\mathbb{R}^+.$ Thus, noting the R-H condition and entropy condition equals 
\begin{equation}\label{eq:E100}
\begin{cases}
u_S(t)=\sqrt{\frac{(\rho_S(t)-\rho_{\infty})(\rho_S^{\gamma}(t)-\rho^{\gamma}_{\infty})}{\rho_{\infty}\rho_S(t)}},\\
u_S(t)=\frac{(\rho_S(t)-\rho_{\infty})}{\rho_S(t)}s'(t),\\
\rho_S(t)>\rho_{\infty},
\end{cases}
\end{equation}
by Theorem \ref{thm:ConstantCase}, we obtain that 
\begin{equation}
    \rho_S(t)\sim u^{\frac{2}{\gamma}}_S(t)\rho_{\infty}^{\frac{1}{\gamma}},
\end{equation}
which together with \eqref{eq:G2} yields that 
\begin{equation}\label{eq:EE4}
\rho_S(t)\sim w'(t)^{\frac{2}{\gamma}}\rho_{\infty}^{\frac{1}{\gamma}}.
\end{equation}
Plugging the first formula in $(\mathcal{H})$ into \eqref{eq:EE4}, the first formula in \eqref{eq:perturbedFormula} follows directly. Moreover, due to $\displaystyle c=\sqrt{\gamma}\rho^{\frac{\gamma-1}{2}},$  the second formula in \eqref{eq:perturbedFormula} is obvious. We next prove the last three formulas in \eqref{eq:perturbedFormula}.

By the definition of $\lambda_\pm$ in \eqref{eq:Characteristics}, there holds that 
\begin{equation}
\lambda_+(x, t)-s'(t)=c(x, t)+\qnt{u(x, t)-s'(t)}
\end{equation}
and
\begin{equation}
\begin{aligned}
    &w'(t)-\lambda_-(x, t)\\&=c(x, t)+\qnt{w'(t)-u(x, t)}\\&=c(x, t)+\qnt{u(w(t), t)-u(x, t)}\\
&=c(x, t)+\qnt{u(w(t), t)-s'(t)}+\qnt{s'(t)-u(x, t)}.
\end{aligned}
\end{equation}
Substituting the obtained estimates on $c$ in the second formula in \eqref{eq:perturbedFormula} and the second formula in  $(\mathcal{H})$ into the former two equations, we obtain the third and the fourth formula in \eqref{eq:perturbedFormula}.

Finally, substituting $\displaystyle \rho(x, t)\sim w'(t)^{\frac{2}{\gamma}}\rho_\infty^{\frac{1}{\gamma}}$ into  \eqref{eq:kgExpan} in Theorem \ref{lem:R3}, we prove the last formula in \eqref{eq:perturbedFormula}. 

Note the definition of the notation $\sim$ in Definition \ref{def:equilinfinitesmall}. The existence of the maps $\mathcal{E}(\sigma)$ and $\mathcal{W}(\sigma)$ is a direct corollary of \eqref{eq:perturbedFormula}.
The proof is complete.
\end{proof}

For $P\in\Omega_T,$ we denote $F\Lambda^P_{\pm}$ and $B\Lambda^P_{\pm}$ as the forward and backward characteristic line passing $P.$ Specifically,
\begin{equation}
    \begin{aligned}
        &F\Lambda^P_{\pm}:=\{(\chi(t),t):\chi'(t)=\lambda_{\pm}(\chi(t),t),\chi(t_p)=r_p, t>t_p\},\\
        &B\Lambda^P_{\pm}:=\{(\chi(t),t):\chi'(t)=\lambda_{\pm}(\chi(t),t),\chi(t_p)=r_p, t<t_p\}.
    \end{aligned}
\end{equation}
We next demonstrate that the characteristics emitted from $\mathsf{P}$ and arriving at $\mathsf{S}$ (or vice versa) lie within a narrow temporal strip. This confinement, in light of \eqref{eq:CharacteristicDecomposition}, ensures that the variations of $\partial^\pm c$ along these characteristics can be effectively bounded by a small quantity.

\begin{lem}[Narrow estimate]\label{lem:r1}
Suppose that \textup{(A1)} holds and that $\displaystyle\qnt{\rho, u}\big|_{\Omega_T} \in C^1(\Omega_T)$ satisfies $(\mathcal{H})$. For any $\displaystyle P(x_p,t_p)\in\Omega_T,$ define $\displaystyle P_\pm(x_p^\pm, t_p^\pm):=F\Lambda^P_{\pm}\cap\partial\Omega_T$. Then, for any $\sigma\in(0,\sqrt{\gamma}{w_*}^{\frac{\gamma-1}{\gamma}})$, it holds that
\begin{equation}\label{eq:narrowont}
t_p^{\pm}-t_p\leq\frac{(\delta_1+{w_*}^{-\frac{1}{\gamma}}{w^*}^{\frac{\gamma-1}{\gamma}})}{\sqrt{\gamma}{w_*}^{\frac{\gamma-1}{\gamma}}-\sigma}\rho_\infty^{\frac{\gamma-1}{2\gamma}}t_p,
\end{equation}
provided that $\displaystyle\rho_\infty\in(0, \mathcal{E}(\sigma))$, where $\mathcal{E}(\sigma)$ is given in Lemma \ref{lem:r1}.
\end{lem}

\begin{proof}
It directly follows from $(\mathcal{H})$ that 
\begin{equation*}
0<s(t_p)-w(t_p)=\int_{0}^{t_p}s'(\tau)-w'(\tau)\dif\tau\leq\qnt{\delta_1t_p+\int_{0}^{t_p}w'(\tau)^{\frac{\gamma-2}{\gamma}}\dif\tau}\rho_\infty^{\frac{\gamma-1}{\gamma}},
\end{equation*}
which together with the boundedness of $w'(\tau)$ in \textup{(A1)} and the fact that $\gamma>1$, yields the following narrow estimate 
\begin{equation}\label{eq:Narrowestimate}
0<s(t_p)-w(t_p)\leq(\delta_1+{w_*}^{-\frac{1}{\gamma}}{w^*}^{\frac{\gamma-1}{\gamma}})\rho_\infty^{\frac{\gamma-1}{\gamma}} t_p.
\end{equation}

Due to the definition of $\displaystyle F\Lambda^P_{\pm},$ it holds that  
\begin{equation}\label{eq:characteristicbound}
\begin{aligned}
&s(t_p^+)-x_p^++\int_{t_p}^{t_p^+}s'(\tau)-\lambda_+(\chi_+(\tau), \tau)\dif\tau=s(t_p)-x_p\leq s(t_p)-w(t_p),\\
&x_p^--w(t_p^-)+\int_{t_p}^{t_p^-}\lambda_-(\chi_-(\tau), \tau)-w'(\tau)\dif\tau=x_p-w(t_p)\leq s(t_p)-w(t_p).
\end{aligned}
\end{equation}
Further substituting the lower bounds of $\displaystyle s'(t)-\lambda_+(x, t)$ and $\displaystyle w'(t)-\lambda_-(x, t)$ given in \eqref{eq:lowerboundedonH} into \eqref{eq:characteristicbound}, and noting
\begin{equation*}
s(t_p^+)-x_p^+>0,\quad x_p^--w(t_p^-)>0,
\end{equation*}
we derive from \eqref{eq:Narrowestimate} and \eqref{eq:characteristicbound} that for any $\sigma\in(0,\sqrt{\gamma}{w_*}^{\frac{\gamma-1}{\gamma}}),$ it holds that  
\begin{equation*}
t_p^{\pm}-t_p\leq  \frac{(\delta_1+{w_*}^{-\frac{1}{\gamma}}{w^*}^{\frac{\gamma-1}{\gamma}})}{\sqrt{\gamma}{w_*}^{\frac{\gamma-1}{\gamma}}-\sigma}\rho_\infty^{\frac{\gamma-1}{2\gamma}}t_p,
\end{equation*}
provided $\displaystyle\rho_\infty\in(0, \mathcal{E}(\sigma)).$  The proof is complete.
\end{proof}

Based on the narrow estimates in Lemma \ref{lem:r1}, by \eqref{eq:CharacteristicDecomposition} along characteristics, we conduct the following a priori estimates.

\begin{lem}\label{lem:Apriori}
Suppose that \textup{(A1)} and \textup{(A3)} hold. There exist positive constants $\delta_1, \delta_2$ and a corresponding positive constant $\epsilon_A$,  such that if $\displaystyle
 \rho_\infty\in(0, \epsilon_A)$ and if $\displaystyle\qnt{\rho, u}\big|_{\Omega_T} \in C^1(\Omega_T)$ satisfies $(\mathcal{H})$, then it holds that $\displaystyle (\rho, u)\Big|_{\Omega_T}$ satisfies $(\tilde{\mathcal{H}})$, with 
\begin{equation*}
(\tilde{\mathcal{H}})
\begin{cases}
\abs{\rho(x, t)-\rho(s(t), t)}\leq \frac{1}{2}\rho_\infty^{\frac{\gamma+1}{2\gamma}},\\
\abs{s'(t)-u(x, t)}\leq (w'(t)^{\frac{\gamma-2}{\gamma}}+\frac{2\delta_1}{3})\rho_\infty^{\frac{\gamma-1}{\gamma}},\\
\max\set{\abs{t\partial^+c},\abs{t\partial^-c}}\leq (1-\frac{\hat{\nu}}{2}\rho_\infty^{\frac{\gamma-1}{2\gamma}})\delta_2\rho_\infty^{\frac{\gamma-1}{2\gamma}},
\end{cases}
\end{equation*}
for some positive constant $\hat{\nu}$ independent of $\rho_\infty$.
\end{lem}
\begin{proof}
For $Q(x_0, t_0)\in\Omega_T,$ we denote 
\begin{equation*}
Q_1(x_1, t_1):=B\Lambda^Q_{+}\cap \mathsf{P},~  Q_2(x_2, t_2):=B\Lambda^{Q_1}_{-}\cap \mathsf{S}.
\end{equation*}
By the narrow estimates \eqref{eq:narrowont}, we have that for any $\sigma\in(0,\sqrt{\gamma}{w_*}^{\frac{\gamma-1}{\gamma}})$,
\begin{equation}\label{eq:EE7}
    \begin{aligned}
        &1\leq\frac{t_0}{t_1}\leq 1+ \frac{(\delta_1+{w_*}^{-\frac{1}{\gamma}}{w^*}^{\frac{\gamma-1}{\gamma}})}{\sqrt{\gamma}{w_*}^{\frac{\gamma-1}{\gamma}}-\sigma}\rho_\infty^{\frac{\gamma-1}{2\gamma}},\\
        & 1\leq {\frac{t_0}{t_2}}\leq\qnt{ 1+ \frac{(\delta_1+{w_*}^{-\frac{1}{\gamma}}{w^*}^{\frac{\gamma-1}{\gamma}})}{\sqrt{\gamma}{w_*}^{\frac{\gamma-1}{\gamma}}-\sigma}\rho_\infty^{\frac{\gamma-1}{2\gamma}}}^2,
    \end{aligned} 
\end{equation}
provided that $\displaystyle\rho_\infty\in(0, \mathcal{E}(\sigma))$.

Integrating the first equation of \eqref{eq:CharacteristicDecomposition} along $B\Lambda^Q_{+}$ and along $B\Lambda^{Q_1}_-$ respectively, we arrive at   
\begin{equation}\label{eq:ee1}
\begin{aligned}
&\partial^-c(Q)=\partial^-c(Q_1)+\int_{t_1}^{t_0} \frac{\gamma+1}{\gamma-1}\frac{1}{2c}(\partial^+c+\partial^-c)\partial^-c\Big|_{B\Lambda^Q_{+}}\dif t,\\
&\partial^+c(Q_1)=\partial^+c(Q_2)+\int_{t_2}^{t_1} \frac{\gamma+1}{\gamma-1}\frac{1}{2c}(\partial^+c+\partial^-c)\partial^+c\Big|_{B\Lambda^{Q_1}_-}\dif t.
\end{aligned}
\end{equation}
Recalling the transformation relations at the piston and the shock given in \eqref{eq:ShockRelations} and \eqref{eq:pistonrelations} respectively, we have that 
\begin{equation}\label{eq:R2}
\begin{aligned}
&\partial^-c(Q_1)=\partial^+c(Q_1)+(\gamma-1)w''(t_1),\\
&\partial^+c(Q_2)=-k_g\partial^-c(Q_2).
\end{aligned}
\end{equation}
Then, combining \eqref{eq:ee1} and \eqref{eq:R2}, we get the relation between $\partial^-c(Q)$ and $\partial^-c(Q_2)$
\begin{equation}\label{eq:EE5}
\begin{split}
\partial^-c(Q)&=-k_g\partial^-c(Q_2)+(\gamma-1)w''(t_1)\\&+\int_{t_2}^{t_1} \frac{\gamma+1}{\gamma-1}\frac{1}{2c}(\partial^+c+\partial^-c)\partial^+c|_{(x, t)=(\chi_-(t), t)}\dif t\\
&+\int_{t_1}^{t_0} \frac{\gamma+1}{\gamma-1}\frac{1}{2c}(\partial^+c+\partial^-c)\partial^-c|_{(x, t)=(\chi_+(t), t)}\dif t.
\end{split}
\end{equation}

Due to $\displaystyle(\rho, u)\Big|_{\Omega_T}$ satisfying $(\mathcal{H})$, the lower bound estimate on sound speed $c$ in \eqref{eq:lowerboundedonH} and the narrow estimates \eqref{eq:narrowont} hold. Thus, substituting \eqref{eq:lowerboundedonH} \eqref{eq:narrowont} and the a priori assumption $\mathcal{H}$ into \eqref{eq:EE5}, a direct computation yields that for any $\displaystyle\sigma_i\in (0, \sqrt{\gamma}{w_*}^{\frac{\gamma-1}{\gamma}}), i=1,2,$
\begin{equation}\label{eq:Longest1}
\begin{aligned}
\abs{\partial^-c(Q)}&\leq \abs{k_g}\delta_2\rho_\infty^{\frac{\gamma-1}{2\gamma}}\frac{1}{t_2} + \delta_2^2 \frac{\gamma+1}{\gamma-1} \frac{1}{\sqrt{\gamma}{w_*}^{\frac{\gamma-1}{\gamma}}-\sigma_2} \frac{(\delta_1+{w_*}^{-\frac{1}{\gamma}}{w^*}^{\frac{\gamma-1}{\gamma}})}{\sqrt{\gamma}{w_*}^{\frac{\gamma-1}{\gamma}}-\sigma_1} \rho_\infty^{\frac{\gamma-1}{\gamma}}\frac{1}{t_2}\\
&\quad+\delta_2^2 \frac{\gamma+1}{\gamma-1} \frac{1}{\sqrt{\gamma}{w_*}^{\frac{\gamma-1}{\gamma}}-\sigma_2} \frac{(\delta_1+{w_*}^{-\frac{1}{\gamma}}{w^*}^{\frac{\gamma-1}{\gamma}})}{\sqrt{\gamma}{w_*}^{\frac{\gamma-1}{\gamma}}-\sigma_1} \rho_\infty^{\frac{\gamma-1}{\gamma}}\frac{1}{t_1} + (\gamma-1)w''(t_1),
\end{aligned}
\end{equation}
provided that $\rho_\infty\in\qnt{0, \min\set{\mathcal{E}(\sigma_1), \mathcal{E}(\sigma_2)}}.$ Specially, we take $\sigma_2=\frac{\sqrt{\gamma}{w_*}^{\frac{\gamma-1}{\gamma}}}{2}$ to simplify \eqref{eq:Longest1} into 
\begin{equation*}
\begin{split}
\abs{\partial^-c(Q)}&\leq \abs{k_g}\frac{\delta_2\rho_\infty^{\frac{\gamma-1}{2\gamma}}}{t_2} + \frac{2\delta_2\qnt{\gamma+1}}{\qnt{\gamma-1}\sqrt{\gamma}{w_*}^{\frac{\gamma-1}{\gamma}}} \frac{(\delta_1+{w_*}^{-\frac{1}{\gamma}}{w^*}^{\frac{\gamma-1}{\gamma}})}{\sqrt{\gamma}{w_*}^{\frac{\gamma-1}{\gamma}}-\sigma_1} \frac{\delta_2\rho_\infty^{\frac{\gamma-1}{\gamma}}}{t_2}\\
&\quad+\frac{2\delta_2\qnt{\gamma+1}}{\qnt{\gamma-1}\sqrt{\gamma}{w_*}^{\frac{\gamma-1}{\gamma}}} \frac{(\delta_1+{w_*}^{-\frac{1}{\gamma}}{w^*}^{\frac{\gamma-1}{\gamma}})}{\sqrt{\gamma}{w_*}^{\frac{\gamma-1}{\gamma}}-\sigma_1} \frac{\delta_2\rho_\infty^{\frac{\gamma-1}{\gamma}}}{t_1} + (\gamma-1)w''(t_1),
\end{split}
\end{equation*}
for $\displaystyle\rho_\infty\in\qnt{0, \min\set{\mathcal{E}(\sigma_1), \mathcal{E}(\frac{\sqrt{\gamma}{w_*}^{\frac{\gamma-1}{\gamma}}}{2})}}.$

Multiplying the former formula by $t_0$, and applying \eqref{eq:EE7}, we obtain that 
\begin{equation}\label{eq:DeriativeRR1}
\begin{split}
\abs{t_0\partial^-c(Q)}&\leq \abs{k_g}\qnt{1+\frac{(\delta_1+{w_*}^{-\frac{1}{\gamma}}{w^*}^{\frac{\gamma-1}{\gamma}})}{\sqrt{\gamma}{w_*}^{\frac{\gamma-1}{\gamma}}-\sigma_1}\rho_\infty^{\frac{\gamma-1}{2\gamma}}}^2\delta_2\rho_\infty^{\frac{\gamma-1}{2\gamma}}\\
&\quad+\frac{2\qnt{\gamma+1}}{\qnt{\gamma-1}\sqrt{\gamma}{w_*}^{\frac{\gamma-1}{\gamma}}} \frac{(\delta_1+{w_*}^{-\frac{1}{\gamma}}{w^*}^{\frac{\gamma-1}{\gamma}})}{\sqrt{\gamma}{w_*}^{\frac{\gamma-1}{\gamma}}-\sigma_1} \qnt{1+\frac{(\delta_1+{w_*}^{-\frac{1}{\gamma}}{w^*}^{\frac{\gamma-1}{\gamma}})}{\sqrt{\gamma}{w_*}^{\frac{\gamma-1}{\gamma}}-\sigma_1}\rho_\infty^{\frac{\gamma-1}{2\gamma}}}\\
&\quad\quad\cdot\qnt{2+\frac{(\delta_1+{w_*}^{-\frac{1}{\gamma}}{w^*}^{\frac{\gamma-1}{\gamma}})}{\sqrt{\gamma}{w_*}^{\frac{\gamma-1}{\gamma}}-\sigma_1}\rho_\infty^{\frac{\gamma-1}{2\gamma}}} \qnt{\delta_2\rho_\infty^{\frac{\gamma-1}{2\gamma}}}^2\\
&\quad+(\gamma-1) \qnt{1+\frac{(\delta_1+{w_*}^{-\frac{1}{\gamma}}{w^*}^{\frac{\gamma-1}{\gamma}})}{\sqrt{\gamma}{w_*}^{\frac{\gamma-1}{\gamma}}-\sigma_1}\rho_\infty^{\frac{\gamma-1}{2\gamma}}} \abs{t_1w''(t_1)}\\
\end{split}
\end{equation}
 for $\displaystyle\rho_\infty\in\qnt{0,\min\set{\mathcal{E}(\sigma_1),\mathcal{E}(\frac{\sqrt{\gamma}{w_*}^{\frac{\gamma-1}{\gamma}}}{2})}}.$ 
 
 Moreover, substituting the upper bound of $k_g$ given in \eqref{eq:kgupperbound} into \eqref{eq:DeriativeRR1} yields that for any $\displaystyle\sigma_1\in(0, \frac{6}{\sqrt{\gamma}}{w^*}^{-\frac{1}{\gamma}})$ and $\displaystyle \sigma_3\in(0, \sqrt{\gamma}{w_*}^{\frac{\gamma-1}{\gamma}})$,
 \begin{equation}\label{eq:Essentialestiamte}
\begin{aligned}
&\frac{\abs{t_0\partial^-c(Q)}}{\delta_2\rho_\infty^{\frac{\gamma-1}{2\gamma}}}\\&\leq 1-\qnt{\frac{6}{\sqrt{\gamma}}{w^*}^{-\frac{1}{\gamma}}-\sigma_3-\frac{2(\delta_1+{w_*}^{-\frac{1}{\gamma}}{w^*}^{\frac{\gamma-1}{\gamma}})}{\sqrt{\gamma}{w_*}^{\frac{\gamma-1}{\gamma}}-\sigma_1}}\rho_\infty^{\frac{\gamma-1}{2\gamma}}\\
&\quad+\qnt{\qnt{\frac{(\delta_1+{w_*}^{-\frac{1}{\gamma}}{w^*}^{\frac{\gamma-1}{\gamma}})}{\sqrt{\gamma}{w_*}^{\frac{\gamma-1}{\gamma}}-\sigma_1}}^2-2\qnt{\frac{6}{\sqrt{\gamma}}{w^*}^{-\frac{1}{\gamma}}-\sigma_3}\qnt{\frac{(\delta_1+{w_*}^{-\frac{1}{\gamma}}{w^*}^{\frac{\gamma-1}{\gamma}})}{\sqrt{\gamma}{w_*}^{\frac{\gamma-1}{\gamma}}-\sigma_1}}}\rho_\infty^{\frac{\gamma-1}{\gamma}}\\
&\quad+\qnt{\frac{6}{\sqrt{\gamma}}{w^*}^{-\frac{1}{\gamma}}-\sigma_3}\qnt{\frac{(\delta_1+{w_*}^{-\frac{1}{\gamma}}{w^*}^{\frac{\gamma-1}{\gamma}})}{\sqrt{\gamma}{w_*}^{\frac{\gamma-1}{\gamma}}-\sigma_1}}^2\rho_\infty^{\frac{3\qnt{\gamma-1}}{2\gamma}}\\
&\quad+\frac{2\qnt{\gamma+1}}{\qnt{\gamma-1}\sqrt{\gamma}{w_*}^{\frac{\gamma-1}{\gamma}}}\frac{(\delta_1+{w_*}^{-\frac{1}{\gamma}}{w^*}^{\frac{\gamma-1}{\gamma}})}{\sqrt{\gamma}{w_*}^{\frac{\gamma-1}{\gamma}}-\sigma_1}\qnt{1+\frac{(\delta_1+{w_*}^{-\frac{1}{\gamma}}{w^*}^{\frac{\gamma-1}{\gamma}})}{\sqrt{\gamma}{w_*}^{\frac{\gamma-1}{\gamma}}-\sigma_1}\rho_\infty^{\frac{\gamma-1}{2\gamma}}}\\
&\quad\quad\cdot\qnt{2+\frac{(\delta_1+{w_*}^{-\frac{1}{\gamma}}{w^*}^{\frac{\gamma-1}{\gamma}})}{\sqrt{\gamma}{w_*}^{\frac{\gamma-1}{\gamma}}-\sigma}\rho_\infty^{\frac{\gamma-1}{2\gamma}}}\delta_2\rho_\infty^{\frac{\gamma-1}{2\gamma}}\\
&\quad+(\gamma-1)\qnt{1+\frac{(\delta_1+{w_*}^{-\frac{1}{\gamma}}{w^*}^{\frac{\gamma-1}{\gamma}})}{\sqrt{\gamma}{w_*}^{\frac{\gamma-1}{\gamma}}-\sigma_1}\rho_\infty^{\frac{\gamma-1}{2\gamma}}}\frac{\abs{t_1w''(t_1)}}{\delta_2\rho_\infty^{\frac{\gamma-1}{2\gamma}}},
\end{aligned}
 \end{equation}
 provided $\displaystyle\rho_\infty\in\qnt{0,\min\{\mathcal{E}(\sigma_1),\mathcal{E}(\frac{\sqrt{\gamma}{w_*}^{\frac{\gamma-1}{\gamma}}}{2}),\mathcal{W}(\sigma_3)\}}.$ 

 Collect the coefficient of $-\rho_\infty^{\frac{\gamma-1}{2\gamma}}$ in \eqref{eq:Essentialestiamte}, and denote it by $\displaystyle\mathcal{T}(\sigma_1,\sigma_3,\delta_1,\delta_2).$ It is obvious that  $\displaystyle\mathcal{T}(\sigma_1,\sigma_3,\delta_1,\delta_2)$ is a continuous function associated to $\displaystyle\sigma_1,\sigma_3,\delta_1,\delta_2.$ Due to \textup{(A1)}, it holds that 
\begin{equation}
\lim_{(\sigma_1,\sigma_3,\delta_1,\delta_2)\to(0+,0+,0+,0+)}\mathcal{T}(\sigma_1,\sigma_3,\delta_1,\delta_2)=\frac{2}{\sqrt{\gamma}}{w^*}^{-\frac{1}{\gamma}}\qnt{3-\frac{{w^*}}{{w_*}}}>0,
\end{equation}
 which together the continuity of $\displaystyle\mathcal{T}(\sigma_1,\sigma_3,\delta_1,\delta_2)$, yields that there exist positive constants $\hat{\nu}\in \mathbb{R}^+$ and 
\begin{equation}
\hat{\sigma}\in(0,\frac{\min\{\sqrt{\gamma}{w_*}^{\frac{\gamma-1}{\gamma}},\frac{6}{\sqrt{\gamma}}{w^*}^{-\frac{1}{\gamma}}\}}{2}),
\end{equation} 
such that 
\begin{equation}
\mathcal{T}(\sigma_1,\sigma_3,\delta_1,\delta_2)>\hat{\nu},
\end{equation}
when $\sigma_1=\sigma_3=\delta_1=\hat{\sigma}$ and $\delta_2\in(0,\hat{\sigma}).$

Therefore, by taking $\sigma_1=\sigma_3=\delta_1=\hat{\sigma}$ and $\delta_2\in(0,\hat{\sigma}),$ it follows from \eqref{eq:Essentialestiamte} that
\begin{equation}\label{eq:Esentialestimate2}
\begin{aligned}
&\frac{\abs{t_0\partial^-c(Q)}}{\delta_2\rho_\infty^{\frac{\gamma-1}{2\gamma}}}\\&\leq 1-\hat{\nu}\rho_\infty^{\frac{\gamma-1}{2\gamma}}
+(\gamma-1)\qnt{1+\frac{(\hat{\sigma}+{w_*}^{-\frac{1}{\gamma}}{w^*}^{\frac{\gamma-1}{\gamma}})}{\sqrt{\gamma}{w_*}^{\frac{\gamma-1}{\gamma}}-\hat{\sigma}}\rho_\infty^{\frac{\gamma-1}{2\gamma}}}\frac{\abs{t_1w''(t_1)}}{\delta_2\rho_\infty^{\frac{\gamma-1}{2\gamma}}}\\
&\quad+\qnt{\qnt{\frac{(\hat{\sigma}+{w_*}^{-\frac{1}{\gamma}}{w^*}^{\frac{\gamma-1}{\gamma}})}{\sqrt{\gamma}{w_*}^{\frac{\gamma-1}{\gamma}}-\hat{\sigma}}}^2-2\qnt{\frac{6}{\sqrt{\gamma}}{w^*}^{-\frac{1}{\gamma}}-\hat{\sigma}}\qnt{\frac{(\hat{\sigma}+{w_*}^{-\frac{1}{\gamma}}{w^*}^{\frac{\gamma-1}{\gamma}})}{\sqrt{\gamma}{w_*}^{\frac{\gamma-1}{\gamma}}-\hat{\sigma}}}}\rho_\infty^{\frac{\gamma-1}{\gamma}}\\
&\quad+\qnt{\frac{6}{\sqrt{\gamma}}{w^*}^{-\frac{1}{\gamma}}-\hat{\sigma}}\qnt{\frac{(\hat{\sigma}+{w_*}^{-\frac{1}{\gamma}}{w^*}^{\frac{\gamma-1}{\gamma}})}{\sqrt{\gamma}{w_*}^{\frac{\gamma-1}{\gamma}}-\hat{\sigma}}}^2\rho_\infty^{\frac{3\qnt{\gamma-1}}{2\gamma}},
\end{aligned}
\end{equation}
for $\rho_\infty\in\qnt{0,\min\{\mathcal{E}(\hat{\sigma}),\mathcal{E}(\frac{\sqrt{\gamma}{w_*}^{\frac{\gamma-1}{\gamma}}}{2}),\mathcal{W}(\hat{\sigma})\}}.$ 

Further substituting the decay assumption on piston $\mathsf{P}$ in \textup{(A3)} into \eqref{eq:Esentialestimate2}, we have that 
\begin{equation}\label{eq:essentialestimate3}
\begin{split}
&\frac{\abs{t_0\partial^-c(Q)}}{\delta_2\rho_\infty^{\frac{\gamma-1}{2\gamma}}}\\&\leq 1-\hat{\nu}\rho_\infty^{\frac{\gamma-1}{2\gamma}}+(\gamma-1)\qnt{1+\frac{(\hat{\sigma}+{w_*}^{-\frac{1}{\gamma}}{w^*}^{\frac{\gamma-1}{\gamma}})}{\sqrt{\gamma}{w_*}^{\frac{\gamma-1}{\gamma}}- \hat{\sigma}}\rho_\infty^{\frac{\gamma-1}{2\gamma}}}\frac{\kappa}{\delta_2}\rho_\infty^{\frac{\gamma-1}{2\gamma}+\varrho}\\
&\quad+\qnt{\qnt{\frac{(\hat{\sigma}+{w_*}^{-\frac{1}{\gamma}}{w^*}^{\frac{\gamma-1}{\gamma}})}{\sqrt{\gamma}{w_*}^{\frac{\gamma-1}{\gamma}}-\hat{\sigma}}}^2-2\qnt{\frac{6}{\sqrt{\gamma}}{w^*}^{-\frac{1}{\gamma}}-\hat{\sigma}}\qnt{\frac{(\hat{\sigma}+{w_*}^{-\frac{1}{\gamma}}{w^*}^{\frac{\gamma-1}{\gamma}})}{\sqrt{\gamma}{w_*}^{\frac{\gamma-1}{\gamma}}-\hat{\sigma}}}}\rho_\infty^{\frac{\gamma-1}{\gamma}}\\
&\quad+\qnt{\frac{6}{\sqrt{\gamma}}{w^*}^{-\frac{1}{\gamma}}-\hat{\sigma}}\qnt{\frac{(\hat{\sigma}+{w_*}^{-\frac{1}{\gamma}}{w^*}^{\frac{\gamma-1}{\gamma}})}{\sqrt{\gamma}{w_*}^{\frac{\gamma-1}{\gamma}}-\hat{\sigma}}}^2\rho_\infty^{\frac{3\qnt{\gamma-1}}{2\gamma}},
\end{split}
\end{equation}
for $\rho_\infty\in\qnt{0,\min\set{\mathcal{E}(\hat{\sigma}),\mathcal{E}(\frac{\sqrt{\gamma}{w_*}^{\frac{\gamma-1}{\gamma}}}{2}),\mathcal{W}(\hat{\sigma})}}.$

 Note that in \eqref{eq:essentialestimate3}, for any fixed $\displaystyle\delta_2\in (0,\hat{\sigma})$, the coefficients of the terms involving $$\displaystyle\rho_\infty^{\frac{\gamma-1}{2\gamma}+\varrho}, \rho_\infty^{\frac{\gamma-1}{\gamma}}, \rho_\infty^{\frac{3\qnt{\gamma-1}}{2\gamma}}$$ are uniformly bounded. We conclude from \eqref{eq:essentialestimate3} that there is a map
\begin{equation*}
\begin{split}
\mathcal{V}:~~(0,\hat{\sigma})&\longrightarrow \Real^+\\
\delta_2&\longrightarrow \mathcal{V}(\delta_2),
\end{split}
\end{equation*}
such that when $\rho_\infty\in\qnt{0,\min\{\mathcal{E}(\hat{\sigma}),\mathcal{E}(\frac{\sqrt{\gamma}{w_*}^{\frac{\gamma-1}{\gamma}}}{2}),\mathcal{W}(\hat{\sigma}),\mathcal{V}(\delta_2)\}},$ there holds that 
\begin{equation}\label{eq:Fianal1}
\frac{\abs{t_0\partial^-c(Q)}}{\delta_2\rho_\infty^{\frac{\gamma-1}{2\gamma}}}\leq 1-\frac{\hat{\nu}}{2}\rho_\infty^{\frac{\gamma-1}{2\gamma}}.
\end{equation}

In the same way, we have 
\begin{equation}\label{eq:final2}
\frac{\abs{t_0\partial^+c(Q)}}{\delta_2\rho_\infty^{\frac{\gamma-1}{2\gamma}}}\leq 1-\frac{\hat{\nu}}{2}\rho_\infty^{\frac{\gamma-1}{2\gamma}},
\end{equation}
for $\rho_\infty\in\qnt{0,\min\{\mathcal{E}(\hat{\sigma}),\mathcal{E}(\frac{\sqrt{\gamma}{w_*}^{\frac{\gamma-1}{\gamma}}}{2}),\mathcal{W}(\hat{\sigma}),\mathcal{V}(\delta_2)\}}.$

We have now nearly completed the estimates on the derivatives. We next turn our attention to the $C^0$ estimate. 

Due to \eqref{eq:1DEulerCharacteristicForm} and \eqref{eq:CharacteristicExpressionsX}, $\displaystyle\partial^{\pm}u$ are specific functions in terms of $\displaystyle\partial^{\pm}c$ and $c$. Moreover, substituting  the lower bound on $c$ in Lemma \eqref{lem:perturbationfirmula} and the upper bounds of $\partial^{\pm}c$ in \eqref{eq:Fianal1} and \eqref{eq:final2}, into these specific expressions, we arrive at that
\begin{equation}\label{eq:EE11}
\begin{aligned}
&\abs{tu_x}\leq \frac{4\delta_2}{\sqrt{\gamma}(\gamma-1){w_*}^{\frac{\gamma-1}{\gamma}}},\\
&\abs{t\rho_x}=\abs{\gamma^{\frac{1}{1-\gamma}}c^{\frac{3-\gamma}{\gamma-1}}c_x}\leq\frac{\delta_2\gamma^{\frac{1}{1-\gamma}}\qnt{\sqrt{\gamma}{w_*}^{\frac{\gamma-1}{\gamma}}+\hat{\sigma}}^{\frac{3-\gamma}{\gamma-1}}}{\sqrt{\gamma}{w_*}^{\frac{\gamma-1}{\gamma}}}\rho_\infty^{\frac{3-\gamma}{2\gamma}},
\end{aligned}
\end{equation}
for $\displaystyle \rho_\infty\in\qnt{0, \min\set{\mathcal{E}(\hat{\sigma}), \mathcal{E}(\frac{\sqrt{\gamma}{w_*}^{\frac{\gamma-1}{\gamma}}}{2}), \mathcal{W}(\hat{\sigma}), \mathcal{V}(\delta_2)}},$  which together with the estimates on the distance between shock and piston in \eqref{eq:Narrowestimate}, yields that 
\begin{equation}\label{eq:hH2}
\begin{split}
\abs{u(x, t) - u(s(t), t)}&\leq \frac{4(\hat{\sigma}+{w_*}^{-\frac{1}{\gamma}}{w^*}^{\frac{\gamma-1}{\gamma}})}{\sqrt{\gamma}(\gamma-1){w_*}^{\frac{\gamma-1}{\gamma}}}\delta_2\rho_\infty^{\frac{\gamma-1}{\gamma}},\\
\abs{\rho(x, t)-\rho(s(t), t)}&\leq \frac{\gamma^{\frac{1}{1-\gamma}}(\hat{\sigma}+{w_*}^{-\frac{1}{\gamma}}{w^*}^{\frac{\gamma-1}{\gamma}})}{\sqrt{\gamma}{w_*}^{\frac{\gamma-1}{\gamma}}\qnt{\sqrt{\gamma}{w_*}^{\frac{\gamma-1}{\gamma}}+\hat{\sigma}}^{\frac{\gamma-3}{\gamma-1}}}\delta_2\rho_\infty^{\frac{\gamma+1}{2\gamma}}.
\end{split}
\end{equation}

We next estimate the difference between $s'(t)$ and $\displaystyle u(s(t),t)$. Due to the R-H condition that 
\begin{equation}
(\rho(s(t),t)-\rho_\infty)s'(t)=\rho(s(t),t) u(s(t),t),
\end{equation}
we have 
\begin{equation}
\begin{aligned}
\abs{s'(t)-u(s(t), t)}=\abs{\frac{\rho_\infty}{\rho(s(t), t)-\rho_\infty}u(s(t), t)}=\abs{\frac{1}{{\rho(s(t),t)}/{\rho_\infty}-1}}\abs{u(s(t), t)}.
\end{aligned}
\end{equation}
Moreover, substituting $$\displaystyle \rho(s(t),t)\sim w'(t)^{\frac{2}{\gamma}}\rho_\infty^\frac{1}{\gamma}$$ given in \eqref{eq:EE4} and $$\abs{u_S(t) - w'(t)}\leq 2(w'(t)^{\frac{\gamma-2}{\gamma}}+\hat{\sigma})\rho_{\infty}^{\frac{\gamma-1}{\gamma}}$$ given in \eqref{eq:G2}, into the former equality, a direct computation yields that 
\begin{equation}
    \abs{s'(t)-u(s(t), t)} \sim w'(t)^{\frac{\gamma-2}{\gamma}}\rho_\infty^\frac{\gamma-1}{\gamma},
\end{equation}
which together with the boundedness of $w'(t)$, gives that there exists a positive constant $\epsilon_u$ such that if $\rho_\infty\in(0, \epsilon_u),$ it holds that 
\begin{equation}\label{eq:hH1}
\abs{s'(t)-u(s(t), t)}\leq (w'(t)+\frac{\hat{\sigma}}{3})\rho_\infty^{\frac{\gamma-1}{\gamma}}.
\end{equation}

Inserting \eqref{eq:hH1} into \eqref{eq:hH2}, we obtain that 
\begin{equation}\label{eq:EE10}
\begin{aligned}
&\abs{u(x, t)-s'(t)}\leq \qnt{w'(t)+\frac{\hat{\sigma}}{3}+\frac{4(\hat{\sigma}+{w_*}^{-\frac{1}{\gamma}}{w^*}^{\frac{\gamma-1}{\gamma}})}{\sqrt{\gamma}(\gamma-1){w_*}^{\frac{\gamma-1}{\gamma}}}\delta_2}\rho_\infty^{\frac{\gamma-1}{\gamma}},\\&
\abs{\rho(x, t)-\rho(s(t), t)}\leq \frac{\gamma^{\frac{1}{1-\gamma}}(\hat{\sigma}+{w_*}^{-\frac{1}{\gamma}}{w^*}^{\frac{\gamma-1}{\gamma}})}{\sqrt{\gamma}{w_*}^{\frac{\gamma-1}{\gamma}}\qnt{\sqrt{\gamma}{w_*}^{\frac{\gamma-1}{\gamma}}+\hat{\sigma}}^{\frac{\gamma-3}{\gamma-1}}}\delta_2\rho_\infty^{\frac{\gamma+1}{2\gamma}},
\end{aligned}
\end{equation}
for $\rho_\infty\in\qnt{0, \min\set{\mathcal{E}(\hat{\sigma}), \mathcal{E}(\frac{\sqrt{\gamma}{w_*}^{\frac{\gamma-1}{\gamma}}}{2}), \mathcal{W}(\hat{\sigma}), \mathcal{V}(\delta_2), \epsilon_u}}.$  Moreover, by taking  
$$
\delta_2=\hat{\delta}:=\min\set{\hat{\sigma}, \frac{\sqrt{\gamma}(\gamma-1){w_*}^{\frac{\gamma-1}{\gamma}}\hat{\sigma}}{12(\hat{\sigma}+{w_*}^{-\frac{1}{\gamma}}{w^*}^{\frac{\gamma-1}{\gamma}})},  \frac{\sqrt{\gamma}{w_*}^{\frac{\gamma-1}{\gamma}}\qnt{\sqrt{\gamma}{w_*}^{\frac{\gamma-1}{\gamma}}+\hat{\sigma}}^{\frac{\gamma-3}{\gamma-1}}}{2\gamma^{\frac{1}{1-\gamma}}(\hat{\sigma}+{w_*}^{-\frac{1}{\gamma}}{w^*}^{\frac{\gamma-1}{\gamma}})}},
$$
the inequalities in \eqref{eq:EE10} imply that when $\displaystyle\rho_\infty\in\qnt{0, \min\set{\mathcal{E}(\hat{\sigma}), \mathcal{E}(\frac{\sqrt{\gamma}{w_*}^{\frac{\gamma-1}{\gamma}}}{2}), \mathcal{W}(\hat{\sigma}), \mathcal{V}(\hat{\delta}), \epsilon_u}},$
\begin{equation}\label{eq:hH3}
    \abs{u(x, t)-s'(t)}\leq (w'(t)+\frac{2\hat{\sigma}}{3})\rho^{\frac{\gamma-1}{\gamma}}_\infty
\end{equation}
and 
\begin{equation}\label{eq:hH333}
    \abs{\rho(x, t)-\rho(s(t), t)}\leq \frac{1}{2}\rho_\infty^{\frac{\gamma+1}{2\gamma}},
\end{equation}
which correspond precisely to the first two formulas in $(\tilde{\mathcal{H}})$. 

Finally, combining \eqref{eq:Fianal1},\eqref{eq:final2}, \eqref{eq:hH3} and \eqref{eq:hH333}, and taking 
\begin{equation*}
\delta_1=\hat{\sigma}, \delta_2=\hat{\delta}, \epsilon_A=\min\set{\mathcal{E}(\hat{\sigma}), \mathcal{E}(\frac{\sqrt{\gamma}{w_*}^{\frac{\gamma-1}{\gamma}}}{2}), \mathcal{W}(\hat{\sigma}), \mathcal{V}(\hat{\delta}), \epsilon_u},
\end{equation*}
we derive $(\tilde{\mathcal{H}})$ for $\displaystyle\rho_\infty\in(0, \epsilon_A)$. The proof is complete. 
\end{proof}



Now we are ready to prove the main Theorem \ref{thm1}.

\begin{proof}[Proof of Theorem \ref{thm1}]

We divide the proof into two steps: first, we demonstrate the local solvability of problem \eqref{eq:MainProblem}; second, we apply the a priori estimates in Lemma \ref{lem:Apriori} to extend the local solution to a global one.

\textbf{Step 1:} (Local solvability). Applying the local existence results of boundary value problems to hyperbolic system in \cite{MR823237}, we have that for any fixed $\displaystyle\rho_\infty\in \mathbb{R}^+,$ 
there exists $\mathrm{t}>0$ associated to $\rho_\infty$ such that problem \eqref{eq:MainProblem} admits $\displaystyle (\rho,u)\in C^1(\Omega_\mathrm{t})$ and $\displaystyle \|(\rho,u)\|_{C^1(\Omega_\mathrm{t})}$ is bounded. 

Set 
$$(\rho_O,u_O):=\lim_{\Omega_{\mathrm{t}}\ni(x,t)\to(0,0)}(\rho,u)(x,t), s_O:=\lim_{t\to0+}s'(t).$$
Due to the asymptotic behavior of the flow states in the constant velocity speed case in \eqref{thm:ConstantCase}, we have that there exists $\epsilon_O>0$ such that for $\displaystyle\rho_\infty\in(0,\epsilon_O)$, there holds that 
\begin{equation}
    \abs{s_O-u_O}\leq (u_O^\frac{\gamma-2}{\gamma}+\frac{\hat{\sigma}}{2})\rho_\infty^{\frac{\gamma-1}{\gamma}}.
\end{equation}
Therefore, noting that $\displaystyle \|(\rho,u)\|_{C^1(\Omega_\mathrm{t})}<+\infty,$ we conclude that for fixed $\rho_\infty\in(0,\epsilon_O),$ there exists $\displaystyle\mathrm{t}_*$ such that problem \eqref{eq:MainProblem} admits $\displaystyle (\rho,u)\in C^1(\Omega_{\mathrm{t}_*})$ satisfying $(\mathcal{H})$ with $\delta_1=\hat{\sigma}, \delta_2=\hat{\delta}$ determined in the proof of Lemma \ref{lem:Apriori}. 

\textbf{Step 2:} (Global solvability). Take $$\epsilon=\min\set{\epsilon_A, \epsilon_O}.$$ According to \eqref{eq:lowerboundedonH} and \eqref{eq:EE11}, $\displaystyle(\rho,u)\in C^1(\Omega_\mathrm{T})$ satisfying $(\tilde{\mathcal{H}})$ ensures  the uniform positive lower bound of $\displaystyle\abs{\lambda_+-\lambda_-}$ and the uniform upper bound of $\displaystyle\|(\rho,u)\|_{C^1(\Omega_\mathrm{T})},$ for any $\mathrm{T}\in\mathbb{R}^+.$ 

Thus, when $\displaystyle\rho_\infty\in (0,\epsilon)$, the a priori estimates Lemma \ref{lem:Apriori} ensures to extend the local $C^1$ flow field and local $C^2$ shock obtained in \textbf{Step 1} to $\Omega_{+\infty}$, and  $(\tilde{\mathcal{H}})$ holds all the time. The proof is complete. 
\end{proof}
 
 \medskip
 
 \section*{Acknowledgments}
 Qianfeng Li was partially supported by Sino-German (CSC-DAAD) Postdoc Scholarship Program, 2023 (No. 57678375). Yongqian Zhang was partially supported by NSFC Project 11421061 and by NSFC Project 12271507.

 \section*{Declarations}

\noindent\textbf{Conflict of interest} 
On behalf of all authors, the corresponding author states that there is no conflict of interest.

\noindent\textbf{Data Availability} The paper does not use any data set.

\bibliographystyle{plain}
\bibliography{CKWX20240603}

\begin{thebibliography}{10}

\bibitem{amadori1997initial}
Debora Amadori.
\newblock Initial-boundary value problems for nonlinear systems of conservation
  laws.
\newblock {\em Nonlinear Differential Equations and Applications NoDEA},
  4(1):1--42, 1997.

\bibitem{KuangJie2021}
Gui-Qiang~G. Chen, Jie Kuang, Wei Xiang, and Yongqian Zhang.
\newblock Hypersonic similarity for steady compressible full {E}uler flows over
  two-dimensional {L}ipschitz wedges.
\newblock {\em Adv. Math.}, 451:Paper No. 109782, 100, 2024.

\bibitem{Chen2003JDE}
Shuxing Chen.
\newblock A singular multi-dimensional piston problem in compressible flow.
\newblock {\em Journal of Differential Equations}, 189(1):292--317, 2003.

\bibitem{Wang2005DCDS}
Shuxing Chen, Gui-Qiang Chen, Zejun Wang, and Dehua Wang.
\newblock A multidimensional piston problem for the euler equations for
  compressible flow.
\newblock {\em Discrete and Continuous Dynamical Systems}, 13(2):361--383,
  2005.

\bibitem{Wangzejun2004}
Shuxing Chen, Zejun Wang, and Yongqian Zhang.
\newblock Global existence of shock front solutions to the axially symmetric
  piston problem for compressible fluids.
\newblock {\em J. Hyperbolic Differ. Equ.}, 1(1):51--84, 2004.

\bibitem{Wangzejun2008Global}
Shuxing Chen, Zejun Wang, and Yongqian Zhang.
\newblock Global existence of shock front solution to axially symmetric piston
  problem in compressible flow.
\newblock {\em Z. Angew. Math. Phys.}, 59(3):434--456, 2008.

\bibitem{Courant1948}
R.~Courant and K.~O. Friedrichs.
\newblock {\em Supersonic {F}low and {S}hock {W}aves}.
\newblock Interscience Publishers, Inc., New York, N. Y., 1948.

\bibitem{cui2009global}
Dacheng Cui and Huicheng Yin.
\newblock Global supersonic conic shock wave for the steady supersonic flow
  past a cone: polytropic gas.
\newblock {\em Journal of Differential Equations}, 246(2):641--669, 2009.

\bibitem{Dingmin2021}
Min Ding.
\newblock Non-relativistic limits of contact discontinuities to 1-{D} piston
  problem for the relativistic full {E}uler system.
\newblock {\em J. Differential Equations}, 274:510--542, 2021.

\bibitem{MR3582280}
Min Ding, Jie Kuang, and Yongqian Zhang.
\newblock Global stability of rarefaction wave to the 1-{D} piston problem for
  the compressible full {E}uler equations.
\newblock {\em J. Math. Anal. Appl.}, 448(2):1228--1264, 2017.

\bibitem{Dingmin2013}
Min Ding and Yachun Li.
\newblock Global existence and non-relativistic global limits of entropy
  solutions to the 1{D} piston problem for the isentropic relativistic {E}uler
  equations.
\newblock {\em J. Math. Phys.}, 54(3):031506, 28, 2013.

\bibitem{ding1985convergenceI}
Xiaxi Ding, Guiqiang Chen, and Peizhu Luo.
\newblock Convergence of the lax-friedrichs scheme for isentropic gas dynamics
  (i).
\newblock {\em Acta Mathematica Scientia}, 5(4):415--432, 1985.

\bibitem{ding1985convergenceII}
Xiaxi Ding, Guiqiang Chen, and Peizhu Luo.
\newblock Convergence of the lax-friedrichs scheme for isentropic gas dynamics
  (ii).
\newblock {\em Acta Mathematica Scientia}, 5(4):433--472, 1985.

\bibitem{diperna1983convergence}
Ronald~J DiPerna.
\newblock Convergence of the viscosity method for isentropic gas dynamics.
\newblock {\em Communications in mathematical physics}, 91:1--30, 1983.

\bibitem{hu2018supersonic}
Dian Hu.
\newblock The supersonic flow past a wedge with large curved boundary.
\newblock {\em Journal of Mathematical Analysis and Applications},
  462(1):380--389, 2018.

\bibitem{hu2025inverse}
Dian Hu, Qianfeng Li, and Yongqian Zhang.
\newblock An inverse problem for multi-dimensional piston models with large
  velocity variations.
\newblock {\em arXiv preprint arXiv:2505.10209}, 2025.

\bibitem{HuZhang2019SIMA}
Dian Hu and Yongqian Zhang.
\newblock Global conic shock wave for the steady supersonic flow past a curved
  cone.
\newblock {\em SIAM J. Math. Anal.}, 51(3):2372--2389, 2019.

\bibitem{Lai2020EJAM}
Geng Lai.
\newblock Self-similar solutions of the radially symmetric relativistic euler
  equations.
\newblock {\em European Journal of Applied Mathematics}, 31(6):919--949, 2020.

\bibitem{Laigeng2022}
Geng Lai.
\newblock Global solution to a three-dimensional spherical piston problem for
  the relativistic {E}uler equations.
\newblock {\em European J. Appl. Math.}, 33(1):1--26, 2022.

\bibitem{lai2023three}
Geng Lai.
\newblock Three-dimensional stationary supersonic flows with axial symmetry in
  relativistic hydrodynamics.
\newblock {\em Journal of the London Mathematical Society}, 108(2):702--741,
  2023.

\bibitem{lefranccois2010introduction}
Emmanuel Lefran{\c{c}}ois and Jean-Paul Boufflet.
\newblock An introduction to fluid-structure interaction: application to the
  piston problem.
\newblock {\em SIAM review}, 52(4):747--767, 2010.

\bibitem{LiYangZheng2011JDE}
Jiequan Li, Zhicheng Yang, and Yuxi Zheng.
\newblock Characteristic decompositions and interactions of rarefaction waves
  of 2-{D} {E}uler equations.
\newblock {\em J. Differential Equations}, 250(2):782--798, 2011.

\bibitem{LiZhangZheng2006CMP}
Jiequan Li, Tong Zhang, and Yuxi Zheng.
\newblock Simple waves and a characteristic decomposition of the two
  dimensional compressible {E}uler equations.
\newblock {\em Comm. Math. Phys.}, 267(1):1--12, 2006.

\bibitem{LiZheng2009ARMA}
Jiequan Li and Yuxi Zheng.
\newblock Interaction of rarefaction waves of the two-dimensional self-similar
  {E}uler equations.
\newblock {\em Arch. Ration. Mech. Anal.}, 193(3):623--657, 2009.

\bibitem{li1994global}
Ta~Tsien Li.
\newblock {\em Global classical solutions for quasilinear hyperbolic systems},
  volume~32 of {\em RAM: Research in Applied Mathematics}.
\newblock Masson, Paris; John Wiley \& Sons, Ltd., Chichester, 1994.

\bibitem{MR823237}
Ta~Tsien Li and Wen~Ci Yu.
\newblock {\em Boundary value problems for quasilinear hyperbolic systems}.
\newblock Duke University Mathematics Series, V. Duke University, Mathematics
  Department, Durham, NC, 1985.

\bibitem{li2007inverse}
Tatsien Li and Libin Wang.
\newblock Inverse piston problem for the system of one-dimensional isentropic
  flow.
\newblock {\em Chinese Annals of Mathematics, Series B}, 28:265--282, 2007.

\bibitem{Li1991GlobalShock}
Tatsien Li and Yan~Chun Zhao.
\newblock Global shock solutions to a class of piston problems for the system
  of one-dimensional isentropic flow.
\newblock {\em Chinese Ann. Math. Ser. B}, 12(4):495--499, 1991.
\newblock A Chinese summary appears in Chinese Ann. Math. Ser. A {{\bf{1}}2}
  (1991), no. 5, 645.

\bibitem{liu1978free}
Tai-Ping Liu.
\newblock The free piston problem for gas dynamics.
\newblock {\em Journal of Differential Equations}, 30(2):175--191, 1978.

\bibitem{MR4466980}
Wen-Jian Peng and Tian-Yi Wang.
\newblock On vanishing pressure limit of continuous solutions to the isentropic
  {E}uler equations.
\newblock {\em J. Hyperbolic Differ. Equ.}, 19(2):311--336, 2022.

\bibitem{MR4052901}
Aifang Qu, Hairong Yuan, and Qin Zhao.
\newblock High {M}ach number limit of one-dimensional piston problem for
  non-isentropic compressible {E}uler equations: polytropic gas.
\newblock {\em J. Math. Phys.}, 61(1):011507, 14, 2020.

\bibitem{MR4887787}
Aifang Qu, Hairong Yuan, and Renxiong Zhao.
\newblock Measure-valued solution for moving piston in pressureless {E}uler
  flows by the method of integration on path space.
\newblock {\em Math. Ann.}, 392(1):1203--1251, 2025.

\bibitem{takeno1992initial}
Shigeharu Takeno.
\newblock Initial boundary value problems for isentropic gas dynamics.
\newblock {\em Proceedings of the Royal Society of Edinburgh Section A:
  Mathematics}, 120(1-2):1--23, 1992.

\bibitem{takeno1995free}
Shigeharu Takeno.
\newblock Free piston problem for isentropic gas dynamics.
\newblock {\em Japan journal of industrial and applied mathematics},
  12:163--194, 1995.

\bibitem{Taylor1946}
G.~I. Taylor.
\newblock The air wave surrounding an expanding sphere.
\newblock {\em Proc. Roy. Soc. London Ser. A}, 186:273--292, 1946.

\bibitem{Tsien1946}
Hsue-shen Tsien.
\newblock Similarity laws of hypersonic flows.
\newblock {\em J. Math. Phys. Mass. Inst. Tech.}, 25:247--251, 1946.

\bibitem{wang2014inverse}
Libin Wang.
\newblock An inverse piston problem for the system of one-dimensional adiabatic
  flow.
\newblock {\em Inverse Problems}, 30(8):085009, 2014.

\bibitem{Wang2004ACTA}
Ze~Jun Wang.
\newblock Local existence of the shock front solution to the axi--symmetrical
  piston problem in compressible flow.
\newblock {\em Acta Mathematica Sinica}, 20(4):589--604, 2004.

\bibitem{Wang2005GlobalExistence}
Zejun Wang.
\newblock Global existence of shock front solutions to the 1-dimensional piston
  problem.
\newblock {\em Annals of Mathematics (Series A)}, 26(4):549--560, 2005.
\newblock In Chinese.

\end{thebibliography}

\end{document}